\documentclass[reqno]{amsart}
\usepackage{lmodern}

\usepackage{amssymb,amsmath,amsfonts,amsthm,epsfig,pstricks,graphics,tikz}

\usepackage{hyperref}
\usepackage{cite}

\usepackage[width=0.85\textwidth ]{caption}

\newtheorem{theorem}{Theorem}[section]

\newtheorem{lemma}[theorem]{Lemma}

\usepackage[foot]{amsaddr}  

\providecommand{\keywords}[1]
{
  \small	
  \textbf{\textit{Keywords---}} #1
}

\title{Gravitational billiards - bouncing inside a paraboloid cavity}

\author{Daniel Jaud}
\thanks{OrcID: 0000-0002-0163-7586}
\address{Gymnasium Holzkirchen, Germany}
\email{Daniel.Jaud.PhD@gmail.com}
\date{}

\begin{document}
\maketitle
\flushbottom


\begin{abstract}
\noindent
In this work the confined domains for a point-like particle propagating within the boundary of an ideally reflecting paraboloid mirror are derived. Thereby it is proven that all consecutive flight parabola foci points lie on the surface of a common sphere of radius $R$. The main results are illustrated in various limiting cases and are compared to its two-dimensional counterpart.
\end{abstract}

\vspace*{0.4cm}

\vspace*{0.4cm}


\section{Introduction}\label{sec:Introduction}
Over the last decades, the dynamics of a point-like particle confined to some domain under the influence of a constant gravitational force, shortly called gravitational billiards, has been studied from various perspectives. Starting from \cite{Lehihet1986} who first performed a numerical study of the simplest imaginable system, the wedge, showing that the system can be integrable for certain angles of the wedge. Further study of the system have e.g. been performed in \cite{Anderson,Korsch1999}. Extensions to other two dimensional boundaries (circular, elliptic, oval) or potentials have been performed in e.g. \cite{Costa2015,Korsch_1991}, who showed that for the quadratic and Coulomb potential the system becomes integrable. In 2015, the first study of the dynamics in a three-dimensional cone were performed by \cite{Langer2015,Langer2015_Thesis} showing that certain quantities of the two-dimensional framework map one-to-one in $\mathbb{R}^3$.

In general, the motion of the particle is highly non-trivial and a neat expression for the trajectory at each time is not accessible. For this reason, following \cite{Masalovich2014,Masalovich2020}, the confined domains for a point like particle bouncing in a parabolic, two-dimensional cavity under the influence of a homogeneous gravitational force were derived through a geometric-analytic approach. Recently associated foci curves and confined domains for bouncing inside general two-dimensional boundaries were obtained in \cite{Jaud2022_2}.

In the following the confined domains for a particle bouncing inside a rotational symmetric paraboloid under the influence of a constant gravitational force parallel to the axis of symmetry is studied. Our analysis will show that some two-dimensional features obtained e.g. in \cite{Masalovich2020,Jaud2022_2,10.1007/978-3-031-13588-0_8} will carry over (in some cases) to the three-dimensional scenario. Due to the additional rotational movement associated to conserved angular momentum along the $z-$direction further restrictions compared to the two-dimensional case will emerge.

The structure of this work presents as follows: in Section \ref{sec:Generalities} we will briefly introduce all necessary assumptions and general ideas that we will benefit from in our later analysis. Before diving into a general analysis, Section \ref{sec:Circle_Reduction} will show that under certain restrictions the three-dimensional particle motion can be reduced to the two dimensional force free case within a circle. In Section \ref{sec:General_Domain}, we will first show that all consecutive flight parabola foci points lie on a sphere of radius $R$. With this result, we derive general formulas for the confined regions depending on the system parameters. For a deeper understanding of the general results and related physics, Section \ref{sec:Limiting_Cases} will derive the associated envelope curves and therefore two-dimensional sections of the rotational confined regions for different values of the sphere radius $R$ as well as (reduced) angular momentum $l_z$. Finally a conclusion and outlook on possible future research topics related to this work is made in Section \ref{sec:conclusion}.

\section{Generalities for the paraboloid billiard}\label{sec:Generalities}
Here some general results for the motion of an particle under the influence of a constant gravitational force within a cavity are stated. All obtained results are direct generalizations from the two-dimensional case already discussed in e.g. \cite{Masalovich2020,Jaud2022_2}.

We are considering the movement for a particle of mass $m$ propagating inside a paraboloid mirror under the influence of the constant gravitational force $\vec{F}=-mg\vec{e}_z$ parallel to the $z-$axis. The equation for the boundary of the paraboloid in Cartesian coordinates $(x,y,z)$ reads
\begin{equation}
M(x,y,z)=z-\frac{x^2+y^2}{4f_M}+f_M=0.
\end{equation}
The focus of the paraboloid is centered at the origin of the coordinate system and $f_M$ denotes the focal length of this ideally reflecting mirror (see Figure \ref{fig:Generalities}).

\begin{figure}[htb]
    \centering
    \includegraphics[scale=0.8]{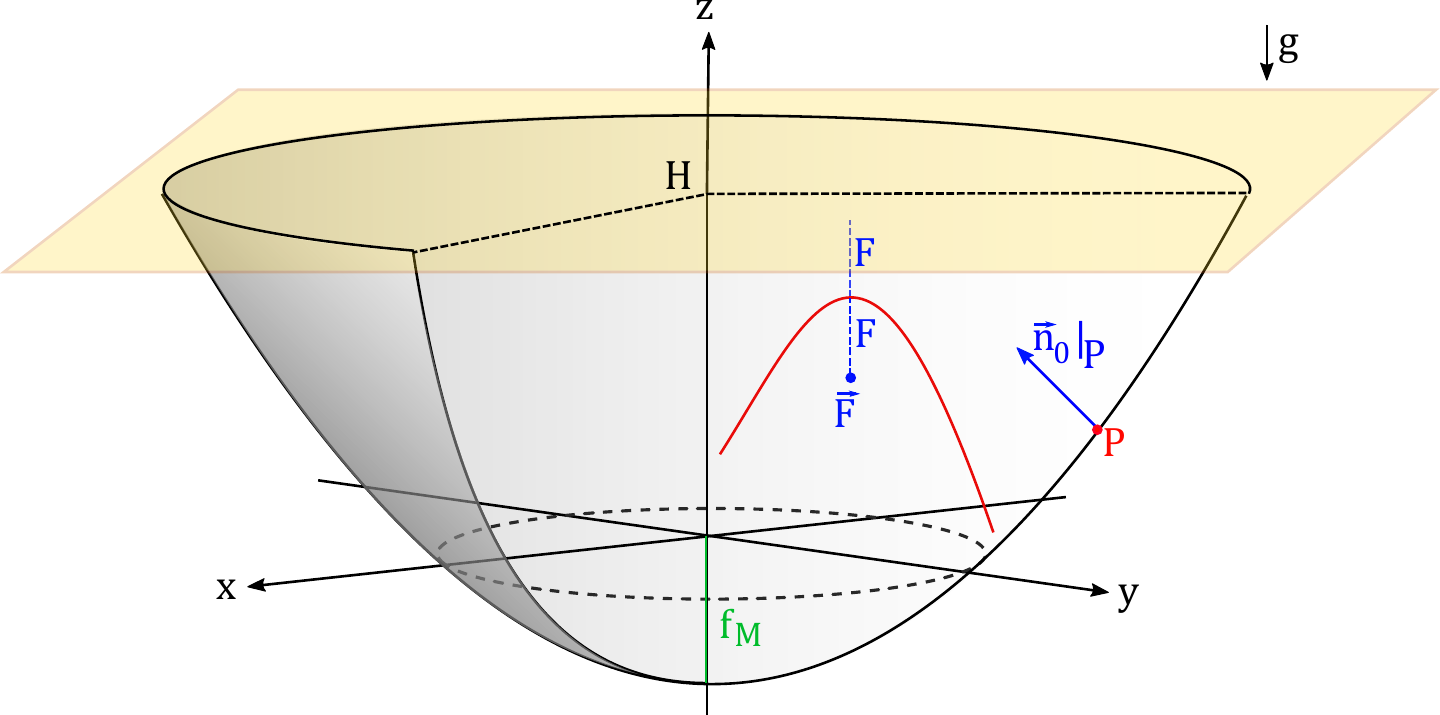}
    \caption{Visualization of main quantities for the paraboloid gravitational billiard.}
    \label{fig:Generalities}
\end{figure}

For a general point $P(x,y,z)$ along the mirror boundary, the associated normalized normal-vector pointing inside the mirror domain is given by
\begin{equation}
\vec{n}_0|_P=\frac{1}{|\vec{\nabla}M|}\vec{\nabla}M|_P=\frac{1}{\sqrt{1+\frac{x^2+y^2}{4f_M^2}}}\begin{pmatrix}
-\frac{x}{2f_M}\\
-\frac{y}{2f_M}\\
1
\end{pmatrix}.
\end{equation}
As usual, the trajectory of one specific flight parabola can be written as a function of time $t$ via
\begin{equation}
\vec{r}(t)=-\frac{1}{2}gt^2 \vec{e}_z+\vec{v}t+\vec{r}_0.
\end{equation}
All flight parabolas posses a focal length $F$ associated with the velocity components in $x-$ and $y-$direction by
\begin{equation}
    F=\frac{v_x^2+v_y^2}{2g}.
\end{equation}
Conservation of energy $E=\frac{m}{2}\vec{v}^2+mgz$ yields a maximal reachable height $H$ for all flight parabolas within the paraboloid. Considering the velocities $\vec{v}_S$ and associated heights $z_S$ at the vertex of a flight parabola we find
\begin{equation}
H=\frac{E}{mg}=\frac{v_S^2}{2g}+z_S=const..
\end{equation}
In analogy to the two-dimensional case (see \cite{Masalovich2020,Jaud2022_2}) $H$ refers to the flight parabola vertex plane (see Figure \ref{fig:Generalities}). As a direct consequence, the $z-$coordinate of the flight parabola focus point $\vec{F}$ fulfills $F_z=H-2F$. In Section \ref{sec:General_Domain} we will make use of this relation.

As a last component we state the law of reflection in vector form whenever the particle hits the boundary of the mirror at a point $P$ and gets ideally reflected. For the velocities $\vec{v}$ before and $\vec{v}'$ after the reflection holds
\begin{equation}\label{eq:law_reflection}
\vec{v}'=\vec{v}-2(\vec{n}_0\circ \vec{v})\cdot \vec{n}_0|_P=\vec{v}-\frac{2}{|\vec{\nabla}M|^2}(\vec{v}\circ \vec{\nabla}M)\cdot \vec{\nabla}M|_P.
\end{equation}
A direct consequence of the law of reflection is stated in the following Lemma which in Section \ref{sec:General_Domain} will be used in order to reduce the number free parameters of our system.
\begin{lemma}{\label{lem:ang_mom_cons}}
Angular momentum per unit mass along the $z-$direction, i.e. $l_z=L_z/m$, in the further course to be called reduced angular momentum, is a conserved quantity in particular at any point $P$ of reflection.
\end{lemma}

\begin{proof}
It is sufficient to proof the statement at some general point  of reflection $P$ associated with the vector $\vec{r}$. Using the law of reflection \eqref{eq:law_reflection} for reduced angular momentum in $z-$direction results in:
$$l_z'=(\vec{r}\times \vec{v}')_z=(\vec{r}\times \vec{v})_z-\frac{2}{|\vec{\nabla}M|^2}(\vec{v}\circ \vec{\nabla}M)\cdot (\vec{r}\times \vec{\nabla}M)_z=(\vec{r}\times \vec{v})_z=l_z.$$

Here we used in the last step the fact, that $(\vec{r}\times \vec{\nabla}M)_z=0$ in $P$.
Therefore $l_z'=x_0v_y-y_0v_x=x_0v_y'-y_0v_x'=l_z$ is conserved. Note that $l_z$ conservation along the flight parabola is a direct consequence by the properties of the cross product.
\end{proof}

\section{Reduction to reflection along the same circle} \label{sec:Circle_Reduction}
In this section first we want to discuss the simplified case in which all consecutive points of reflection $P_i$ lie on a common circle of radius $r_0$ (and consequently height $z_0=\frac{1}{4f_M}r_0^2-f_M$) with respect to the $z-$axis. Naturally for this section polar coordinates are chosen to describe the dynamics. When viewed from above, the system can uniquely be described by the angle $\vartheta$ enclosed by two consecutive points of reflection and the 'origin' at height $z_0$ (see Figure \ref{fig:circle_non_periodic}). Without loss of generality our starting position may be chosen in polar coordinates at $P_0(r_0,0,z_0)$ and consequently $P_i(r_0,i\cdot \vartheta,z_0)$. We choose our particle, when viewed from above, traveling in counter-clockwise direction. The velocity values for the new flight parabola at the point of reflections are given by $(v_{r,i},v_{\varphi,i},v_{z,i})$.

From our setup it is clear that the allowed values for $(v_{r,i},v_{\varphi,i},v_{z,i})$  are restricted by the condition that all point of reflection $P_i$ have to lie on the same circle. In particular, such kind of behavior can only exist if the associated flight parabolas are each a copy of the same fundamental, symmetric, parabola up to an rotation by $\vartheta$ spanned by the two initial reflection points $P_0$ and $P_1$.

\begin{figure}[htb]
    \centering
    \includegraphics[scale=0.65]{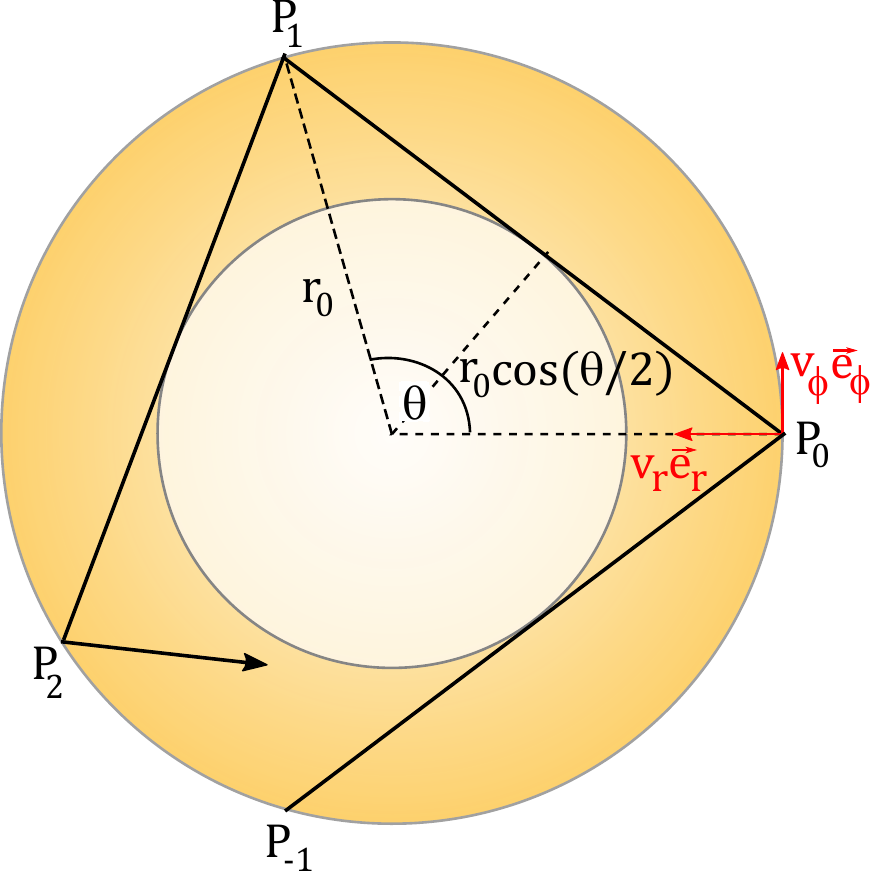}
    \caption{Projection to the parallel $x-y-$plane at height $z_0$ with relevant system parameter $\vartheta$ associated to force free billiards inside the circle.}
    \label{fig:circle_non_periodic}
\end{figure}

Due to rotational symmetry it thus is sufficient to determine the restrictions on $(v_{r,0},v_{\varphi,0},v_{z,0})=:(v_{r},v_{\varphi},v_{z})$. Applying the law of reflection \eqref{eq:law_reflection}
at $P_0$ yields expressions $(v_r',v_\varphi',v_z')$ right before the reflection. Those velocity values correspond to the flight parabola starting at $P_{-1}$ propagating to $P_0$. For all flight parabolas being the same copy one thus obtains the restriction
\begin{equation}
    |v_r|=|v_r'|~~~~\mbox{and}~~~~|v_z|=|v_z'|.
\end{equation}
Both conditions are fulfilled if 
\begin{equation}
    v_r\vec{e}_r+v_z\vec{e}_z\parallel \vec{\nabla}M|_{P_0},
\end{equation}
i.e. the $(r,z)-$components of the velocity vector stand perpendicular on the tangent plane at the point of reflection. The flight time $t=\frac{2v_z}{g}$ for reaching the initial height $z_0$ again is uniquely determined by the motion in $z-$direction. Within this time the particle starting at $P_0$ has to reach $P_1$. In the $x-y-$plane, the angle $\vartheta$ between two consecutive points of reflection (see Figure \ref{fig:circle_non_periodic}) is related to the velocity values $(v_r,v_\varphi)$ via
\begin{equation}
    \vartheta =\pi-2\arctan\left(\left|\frac{v_\varphi}{v_r}\right|\right).
\end{equation}
Demanding for the reflection points all to lie on the same circle of radius $r_0$ gives a further restriction on the system within the given flight time $t$ from $P_i$ to $P_{i+1}$. Direct calculation shows that the allowed velocity components are completely determined by the angle $\vartheta$, the radius $r_0$ of the common reflection points circle as well as the focal length $f_M$ of the paraboloid mirror:
\begin{align}
    v_r&=-\frac{r_0}{2f_M}\cdot \sqrt{gf_M\cdot \left[1-\cos\left(\vartheta\right)\right]},\\
    v_\varphi&=\frac{r_0}{2f_M}\cdot \sqrt{gf_M\cdot \left[1+\cos\left(\vartheta\right)\right]},\\ 
    v_z&=\sqrt{gf_M\cdot \left[1-\cos\left(\vartheta\right)\right]}.
\end{align}
Depending on the values for $\vartheta$  (see. e.g. \cite{Rozikov:458598,10.1007/978-3-031-13588-0_8}) we obtain periodic or non-periodic orbits, where all flight parabolas lie on a common rotational surface around the $z-$axis (see Figure \ref{fig:circle_case_2}) whose radial function is purely determined by the mirror parameters
\begin{equation}
    g(r)=\frac{r_0^2}{4f_M}-\frac{f_M\cdot r^2}{r_0^2},~~~\mbox{with}~~~ r\in [r_0\cos(\vartheta/2);r_0].
\end{equation}
Considering the flight parabolas dividing the rotational flight surface $g(r)$ consecutively into smaller sub regions, one can map this to the circle case as shown in \cite{10.1007/978-3-031-13588-0_8} that for specific values of $\vartheta$ the surface division sequence is given by an integer series.

As a remark for $\vartheta=\pi$ the $\varphi-$velocity component equals zero, i.e. $v_\varphi=0$. For this case there is no rotational motion (angular momentum being zero) and the particle bounces along $g(r)$ forming a two periodic orbit reproducing 
two-dimensional results obtained in e.g. \cite{Korsch_1991,Jaud2022_2}.

\begin{figure}
    \centering
    \includegraphics[scale=0.6]{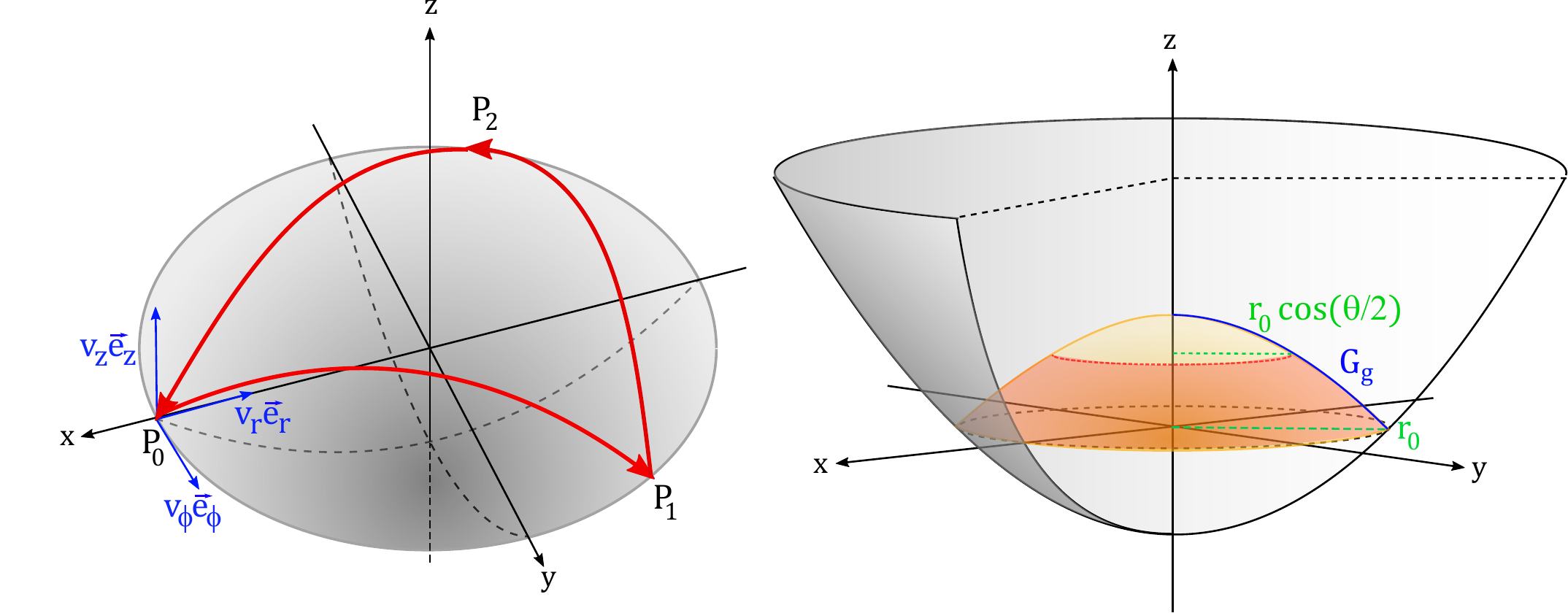}
    \caption{\textit{left}: Example for 3-periodic orbit ($\vartheta=\frac{2\pi}{3}$) along the same circle. \textit{right}: Swept out flight parabola surface (light red) for non-periodic case and $\vartheta \neq \pi$.}
    \label{fig:circle_case_2}
\end{figure}

\newpage
\section{General flight parabola domain}\label{sec:General_Domain}
In this section now we want to derive the confined domains which the particle at a given initial condition can not leave during its motion. We will use the same notations as introduced in Section \ref{sec:Generalities}. It is clear that for certain choices of initial conditions the actual flight orbits will not fill out the entire confined domains. In particular we are considering non periodic orbits in which the swept out region becomes dense. A main component for obtaining expressions for the confined domains is stated in the following theorem.

\begin{theorem}{\label{thm:foci_sphere}}
All flight parabola foci $\vec{F}_i$ at given $H$ and $l_z$ lie on the surface of a sphere with radius $R$ centered at the origin $O$, i.e. the focus of the paraboloid.
\end{theorem}

\begin{proof}
Without loss of generality consider let the point of reflection being located in Cartesian coordinates at $P(r_0,0,z_0)$ and the associated velocity vector right before the reflection take the form $\vec{v}=(v_x,v_y,v_z)^T$. The flight parabola focus then is given by
\begin{equation}
    \vec{F}=\begin{pmatrix}
        r_0\\
        0\\
        2z_0-H
    \end{pmatrix}+\frac{v_z}{g}
    \begin{pmatrix}
        v_x\\
        v_y\\
        v_z
    \end{pmatrix}.   
\end{equation}
Applying the law of reflection in $P$ yields the velocity vector $\vec{v}'$ after the reflection and consequently $\vec{F}'$. A direct but lengthy calculation shows that 
$$|\vec{F}|^2=|\vec{F}'|^2=R^2=const.,$$
i.e. both foci lie on a common sphere of radius $R$. Due to rotational symmetry of the system, conservation of energy as well as reduced angular momentum conservation in $z-$direction, all consecutive foci lie on the same sphere of radius $R$. Since $\vec{r}_0$ and the initial velocity $\vec{v}_0$ have been chosen arbitrarily (under assumption of same total energy) all consecutive foci have to saturate this equality. Note that this is a direct generalization of the two-dimensional case.


\end{proof}
Now we are in the position to reduce the six-dimensional phase space with the knowledge of Theorem \ref{thm:foci_sphere}, conservation of energy and reduced angular momentum as well as rotational invariance, to two free parameters corresponding to confined domains which the particle at given values $(H,l_z,R)$ cannot leave.

The vertex $\vec{S}$ of each flight parabola in spherical coordinates thus can be written as
\begin{equation}
    \vec{S}(R,\varphi,\vartheta)=\vec{F}(R,\varphi,\vartheta)+F\vec{e}_z=R\begin{pmatrix}
        \cos(\varphi)\sin(\vartheta)\\
        \sin(\varphi)\sin(\vartheta)\\
        \cos(\vartheta)
    \end{pmatrix}+F\vec{e}_z=\begin{pmatrix}
        R\cos(\varphi)\sin(\vartheta)\\
        R\sin(\varphi)\sin(\vartheta)\\
        \frac{H+R\cos(\vartheta)}{2}
        \end{pmatrix}.
\end{equation}
Thereby we used that for the focal length of Section \ref{sec:Generalities} holds $F=\frac{H-R\cos(\vartheta)}{2}$. Note that $\vartheta$ in this case corresponds to the polar angle (compare Figure \ref{fig:3d_parameters}) in contrast to the definition for $\vartheta$ of Section \ref{sec:Circle_Reduction}. Energy conservation yields an expression for the absolute value of the velocity $|\vec{v}_S|$ at the vertex
\begin{equation}
   H=\frac{\vec{v}_S^2}{2g}+z_S ~~\leftrightarrow ~~ v_S=|\vec{v}_S|=\sqrt{2g(H-z_S)}=\sqrt{g(H-R\cos(\vartheta)},
\end{equation}
where $H$ is the height of the directrix plane and $z_S$ is the associated vertex height (consider Figure \ref{fig:3d_parameters}). We choose $v_S$ to be positive; negative values simply correspond to a time inverted system. Since the orientation of $\vec{v}_S$ is not fixed by the  equation above we may take a general ansatz $\vec{v}_S=v_S (\cos(\varphi'),\sin(\varphi'),0)^T$. Reduced angular momentum conservation along the $z-$axis (compare Lemma \ref{lem:ang_mom_cons})
\begin{equation}
    l_z=(\vec{S}\times \vec{v}_S)_z=R\sqrt{g(H-R\cos(\vartheta))}\cdot \sin(\vartheta)\cdot \sin(\varphi'-\varphi)=const.,
\end{equation}
restricts the allowed values for the orientation of $\vec{v}_S$ related to the rotation angle $\varphi'$ as follows
\begin{equation}
    \varphi'=\varphi+\arcsin\left(\frac{l_z}{R\sin(\vartheta)\cdot \sqrt{g(H-R\cos(\vartheta))}}\right).
\end{equation}

\begin{figure}[htb]
    \centering
    \includegraphics{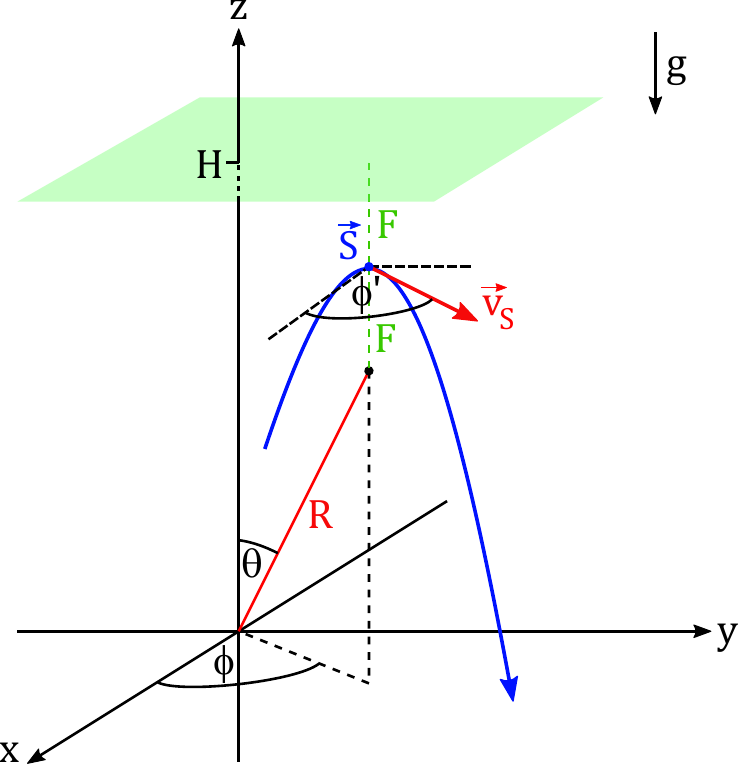}
    \caption{Flight parabola setup.}
    \label{fig:3d_parameters}
\end{figure}

Due to rotational symmetry it is sufficient to consider the case $\varphi=0$ from here on. The corresponding allowed flight parabolas
\begin{equation}\label{eq:parabola_functions}
    \vec{r}(t,\vartheta)=\begin{pmatrix}
        R\sin(\vartheta)+t\cdot \sqrt{g(H-R\cos(\vartheta))-\frac{l_z^2}{R^2\sin^2(\vartheta)}}\\
        t\cdot \frac{l_z}{R\sin(\vartheta)}\\
        -\frac{1}{2}gt^2+\frac{H+R\cos(\vartheta)}{2}
    \end{pmatrix},
\end{equation}
at fixed $(H,l_z,R)$ form a one-parameter family of curves in $\mathbb{R}^3$. Note that the $x-$component velocity term restricts the allowed values for $\vartheta$ at a given value of $l_z$ according to (compare Figure \ref{fig:theta_restriction_given_lz})
\begin{equation}\label{eq:theta_restriction}
   J(H,R,\vartheta)= gR^2\sin^2(\vartheta)\cdot (H-R\cos(\vartheta))\geq l_z^2~~~~\leftrightarrow ~~~~ \vartheta \in [\vartheta_0;\vartheta_1].
\end{equation}
Clearly $l_z$ is the main limiting factor to $\vartheta$ with large $l_z$ associated to a motion farther away from the $z-$axis as in the case for $l_z$ being small, resulting in the possibility of approaching the $z-$axis. Further $l_z$ is bound from above by the maximum of $J(H,R,l_z)$ which is saturated for
\begin{equation}\label{eq:theta_max}
    \cos(\vartheta_{max})=\frac{H- \sqrt{H^2+3R^2}}{3R},
\end{equation}
and therefore takes the value
\begin{equation}\label{eq:J_max}
    J(H,R,\vartheta_{max})=\frac{2}{27}g\left(\sqrt{H^2+3R^2}+2H\right)\left(H\left(\sqrt{H^2+3R^2}-H\right)+3R^2\right).
\end{equation}
\begin{figure}[htb]
    \centering
    \includegraphics[scale=0.7]{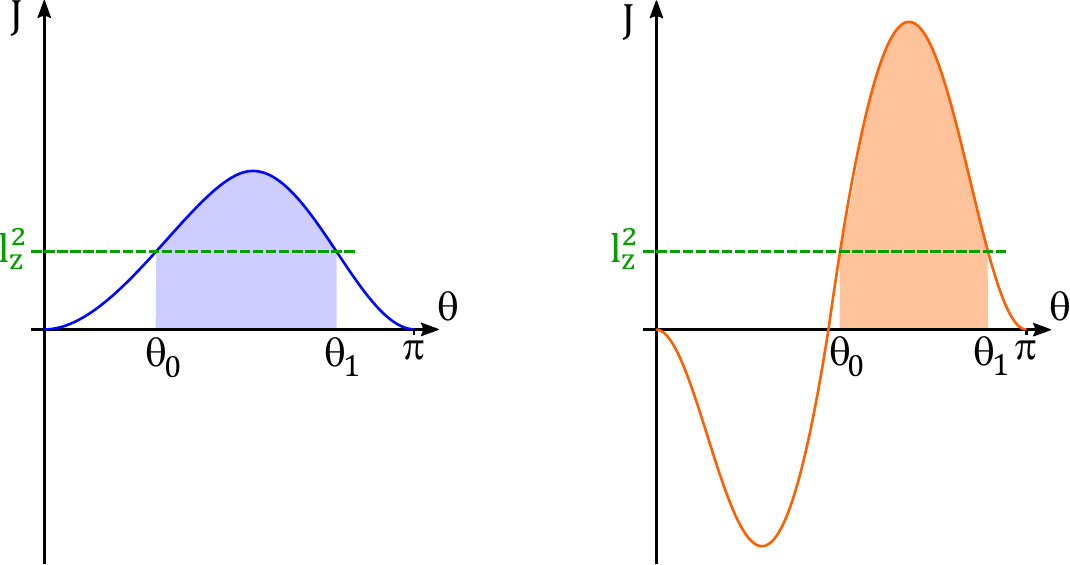}
    \caption{Qualitative restriction of allowed $\vartheta$ values at given $l_z^2$. \textit{left} for $H>R> 0$ and \textit{right} for $0<H<R$.}
    \label{fig:theta_restriction_given_lz}
\end{figure}

\begin{theorem}{\label{thm:height_fct}}
The rotational symmetric allowed propagation heights $h$ of the particle at a given radial distance $r$ and angle $\vartheta$ are given by
$$\label{eq:height_functions}\textstyle
    h_\pm(r,\vartheta)=\frac{H+R\cos(\vartheta)}{2}-\frac{g}{2}\left(\frac{\sqrt{r^2\cdot g(H-R\cos(\vartheta))-l_z^2}\pm \sqrt{R^2\sin^2(\vartheta)\cdot g(H-R\cos(\vartheta))-l_z^2}}{g(H-R\cos(\vartheta))}\right)^2,
$$
with the restriction on the radius $r\geq \frac{l_z}{\sqrt{g(H-R\cos(\vartheta))}}$.
\end{theorem}

\begin{proof}
Equation \eqref{eq:parabola_functions} defines possible trajectories at given $(H,R,l_z)$. Considering the associated radial distance $r^2=x(t,\vartheta)^2+y(t,\vartheta)^2$ one can express the $t$ variable in terms of $r$ and $\vartheta$. Inserting this expression into the $z-$component of \eqref{eq:parabola_functions} yields the expression for the allowed heights $h_\pm$, where the two different solutions correspond to the left and right parabola arc measured from the minimal distance $r_{min}=l_z/\sqrt{g(H-R\cos(\vartheta))}$. 
\end{proof}

In order to obtain expressions for the associated envelope curves we define a new quantity
\begin{equation}\label{eq:envelope_defining_function}
K(z,r,\vartheta,H,R,l_z):= z-h_\pm(r,\vartheta).
\end{equation}
The envelope curves restricting the confined domains then are obtained eliminating $\vartheta$ by solving the following system of equations (see 
\cite{bruce_giblin_1992})
\begin{equation} \label{eq:envelope}
K(z,r,\vartheta,H,R,l_z)=0~~~~~~~\mbox{and} ~~~~~~~    \frac{\partial K}{\partial \vartheta}=0.
\end{equation}
A computer animated picture of allowed flight parabolas is shown in Figure \ref{fig:plane_domains}. In the next section our obtained results will be illustrated in various extremal limits.

\begin{figure}[htb]
    \centering
    \includegraphics[scale=0.6]{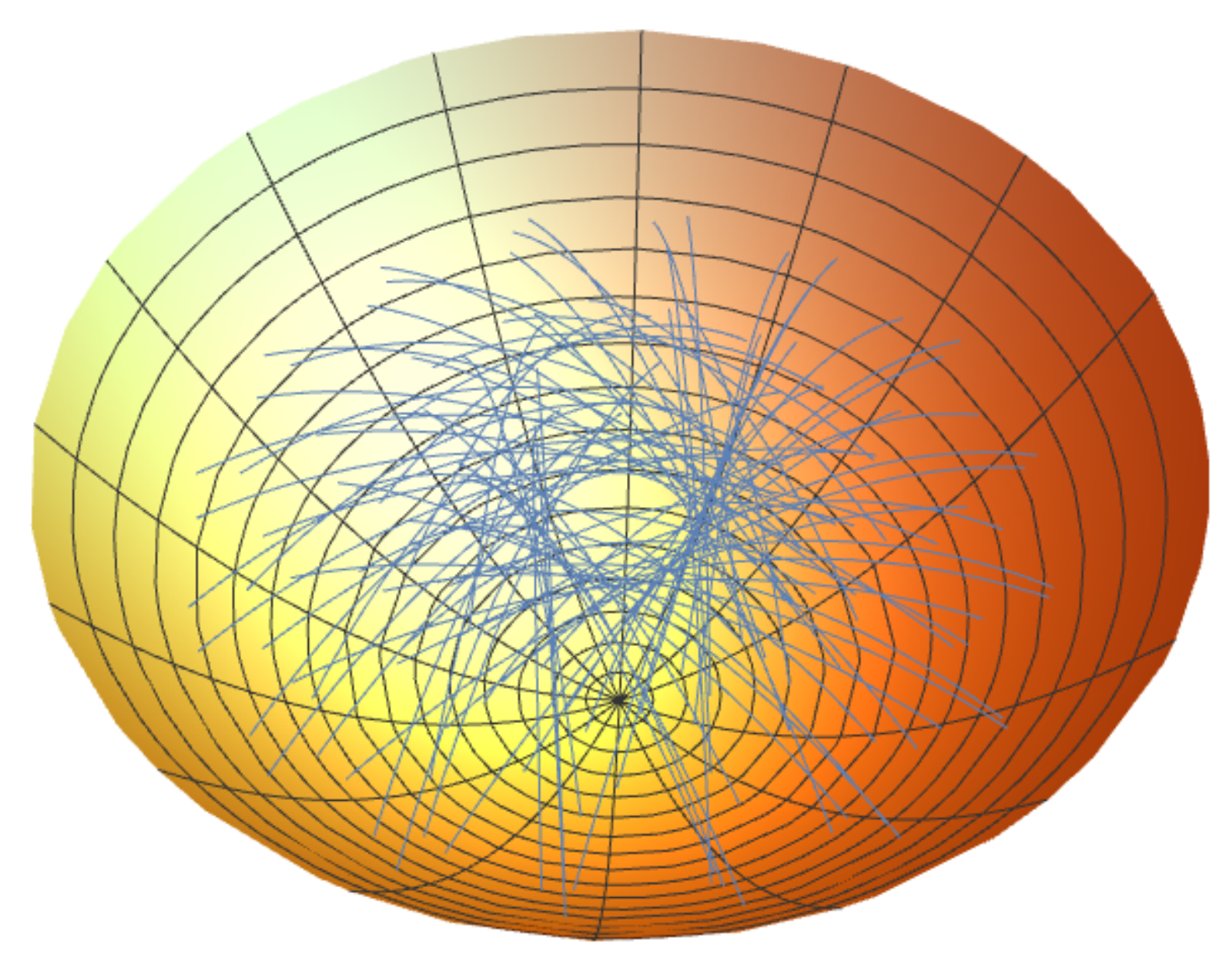}
    \caption{Computer animated flight trajectories.}
    \label{fig:plane_domains}
\end{figure}

\newpage
\section{Discussion of limiting cases}\label{sec:Limiting_Cases}

In this section four limiting cases in terms of the reduced angular momentum $l_z$ are discussed. For all cases we determine the associated height function from Theorem \ref{thm:height_fct} and calculate, if possible, the corresponding envelope curves restricting the motion of the particle, in general, to a rotational symmetric region. All results of this section are displayed in Figure \ref{fig:confined_domains} for illustrative purposes.

\subsection{The $l_z=0$ case}
In the simplest case of no reduced angular momentum ($l_z=0$) the height functions of Theorem \ref{thm:height_fct} significantly simplify to 
\begin{equation}
   h_\pm(r,\vartheta)= \frac{H+R\cos(\vartheta)}{2}-\frac{(r\pm R\sin(\vartheta))^2}{2(H-R\cos(\vartheta))}.
\end{equation}
 Solving the system \eqref{eq:envelope} yields the envelope curves
\begin{equation}
    c_\pm(r)=\frac{H\pm R}{2}-\frac{r^2}{2(H\pm R)}.
\end{equation}
This reproduces the results obtained geometrically in \cite{Masalovich2020} and analytically in \cite{Jaud2022_2}. Since the motion lies in a common plane containing the $z-$axis it is clear that $r$ can take values in $\mathbb{R}$.

\subsection{The small $l_z$ case}

For $l_z$ small the deviation from the $l_z=0$ case is marginal. Thus one can conclude that in first approximation one obtains the same envelope curves $c_\pm(r)$ as before. An additional restriction comes from the fact that the allowed values for $r$ are bound from below by $r\geq \frac{l_z}{\sqrt{g(H-R\cos(\vartheta))}}$. If this inequality is saturated, i.e. we consider the case of minimal radial distance in terms of $\vartheta$, one can solve $r= \frac{l_z}{\sqrt{g(H-R\cos(\vartheta))}}$ for $\cos(\vartheta)$ and insert this expression into the height functions of Theorem \ref{thm:height_fct} yielding one additional (approximate) envelope curve associated with the angular momentum barrier as
\begin{equation}
    c_0(r)= \frac{(H^2-R^2)g}{2l_z^2}\cdot r^2+\frac{gr^4}{2l_z^4}.
\end{equation}
This envelope curve is reminiscent of the Higgs-potential in particle physics, in which in the cases $R>H$ one obtains the well known Mexican-hat like function.

\subsection{The large $l_z$ case}
In the large $l_z$ limit the second square root appearing in the height functions of Theorem \ref{thm:height_fct} in lowest order can be neglected since $l_z^2\approx J(H,R,\vartheta_{max})$. The associated envelope curves thus approximately resemble the height functions for small variations of $\vartheta$ 
\begin{equation}
    \tilde{c}_\pm =\frac{H+R\cos(\vartheta)}{2}-\frac{r^2-R^2\sin^2(\vartheta_{max})}{2(H-R\cos(\vartheta))},
\end{equation}
where $\vartheta \in [\vartheta_{max}-\delta;\vartheta_{max}+\delta]$ for $\delta$ small.

\subsection{The maximal $l_z$ case}
The maximal value for $l_z$ follows from \eqref{eq:J_max} and is given by
\begin{equation}
    l_z=\sqrt{J(H,R,\vartheta_{max})}.
\end{equation}
In these cases, the second square root for the height function of Theorem \ref{thm:height_fct} vanishes, resulting in a single height function for $\vartheta=\vartheta_{max}$ as
\begin{equation}
    d(r)=\frac{H+R\cos(\vartheta_{max})}{2}-\frac{r^2-R^2\sin^2(\vartheta_{max})}{2(H-R\cos(\vartheta_{max}))},
\end{equation}
with $r\geq R\sin(\vartheta_{max})$. Note that for $R<H$ this reproduces the results of Section \ref{sec:Circle_Reduction}. For $R>H$ it depends on the mirror boundary if the condition $r\geq R\sin(\vartheta_{max})$ can be saturated, cases exist in which the maximal $l_z-$value is not accessible due to the mirror boundary.

\begin{figure}[htb]
    \centering
    \includegraphics[scale=0.55]{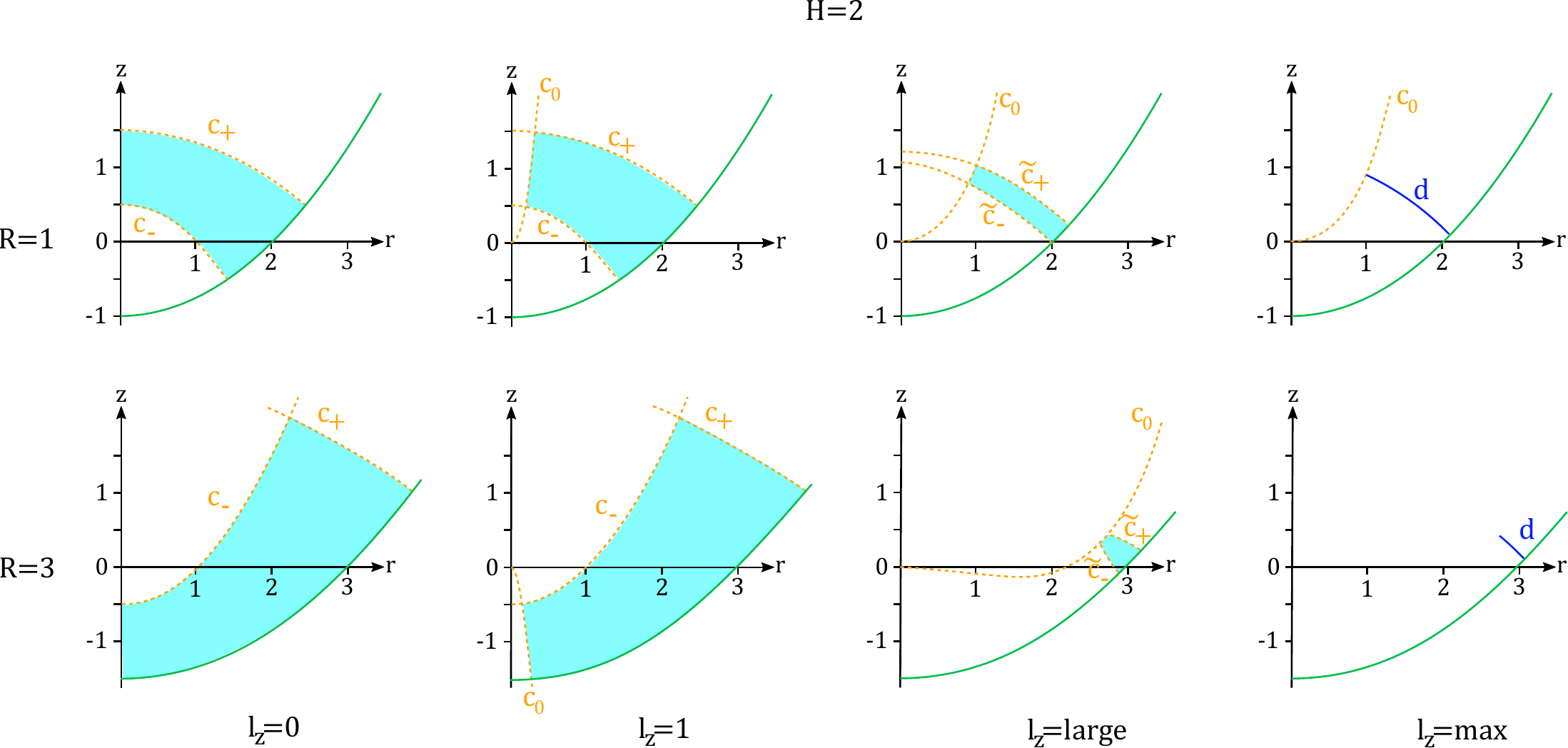}
    \caption{Two-dimensional section of confined domains associated to the four discussed limiting cases. The three-dimensional confined regions are obtained by rotation around the $z-$axis.}
    \label{fig:confined_domains}
\end{figure}


\newpage
\section{Conclusion and Outlook}\label{sec:conclusion}
In this work the rotational symmetric confined domains of a point-like particle bouncing inside a paraboloid cavity under the influence of a homogeneous gravitational field in terms of the directrix height $H$, reduced angular momentum $l_z$ and foci sphere radius $R$ were derived. It has been shown that some two-dimensional results map one to one to the 3D case. In addition, reduced angular momentum conservation (absent in 2D) yields some additional physics in 3D. 

For future works it would be interesting to generalize our results to other rotational symmetric domains. Also, the motion in a non-constant, e.g. Coulomb-field, would be of interest.
%
\vspace*{1cm}
\section*{Acknowledgements}
We would like to thank Dan Reznik for the inspiring conversation leading to this work.
%



\vspace*{1cm}



\end{document}